\documentclass[12pt,a4paper,reqno]{amsproc}
\usepackage{amsmath, amsthm, amscd, amsfonts, amssymb, graphicx, color}
\usepackage{subcaption}
\usepackage[bookmarksnumbered, colorlinks, plainpages]{hyperref}
\hypersetup{colorlinks=true,linkcolor=red, anchorcolor=green, citecolor=cyan, urlcolor=red, filecolor=magenta, pdftoolbar=true}

\newtheorem{theorem}{Theorem}[section]
\newtheorem{lemma}[theorem]{Lemma}
\newtheorem{proposition}[theorem]{Proposition}
\newtheorem{corollary}[theorem]{Corollary}
\theoremstyle{definition}
\newtheorem{definition}[theorem]{Definition}

\theoremstyle{remark}

\numberwithin{equation}{section}
\newcommand{\Tau}{\mathrm{T}}

\newtheorem*{theorem-non}{Theorem}

\usepackage{enumerate}
\begin{document}
\setcounter{page}{1}
\begin{abstract}
In this paper, we introduce a pseudo-differential operator associated with the linear canonical Dunkl transform. For a particular class of symbols, we prove that this operator defines a continuous linear mapping from the Schwartz space into itself. We further establish an amplitude integral and kernel representation of the operator and investigate the smoothness and decay properties of the associated kernel. Subsequently, we establish an 
$L^1$-norm inequality for the pseudo-differential operator on the linear canonical Dunkl Sobolev spaces. Moreover, we extend the pseudo-differential operator to tempered distributions, proving its continuity. We then investigate the  $L^2$-boundedness of the operator for the symbol class 
$S^m_0$. Finally, as an application, we employ the pseudo-differential operator to study non-homogeneous and nonlinear parabolic partial differential equations.
\end{abstract}
\title[ Pseudo differential operator]{Pseudo-differential operator associated with the linear canonical Dunkl transform}
\author[S. Umamaheswari  and Sandeep Kumar Verma]{S. Umamaheswari  \and Sandeep Kumar Verma}

\address{Department of Mathematics, SRM University-AP, Andhra Pradesh, Guntur--522240, India}

 \email{umasmaheswari98@gmail.com, sandeep16.iitism@gmail.com}




\maketitle
\section{Introduction} 
The theory of pseudo-differential operators emerged as a natural extension of classical differential and singular integral operators in the study of partial differential equations. The fundamental solutions of elliptic operators were written as singular integral kernels; however, this approach provided limited clarity regarding their structure. Researchers, notably Friedrichs and later Kohn, Nirenberg \cite{Kohn}, and H\"ormander \cite{Hormander}, realized that such operators could be expressed more effectively through Fourier multipliers, leading to the concept of pseudo-differential operators. This framework enabled the handling of variable coefficient equations, non-elliptic problems, and boundary value problems with far greater generality. By encoding both local and frequency behaviour of functions in their symbolic calculus, pseudo-differential operators provided a powerful and flexible tool for analyzing regularity, hypoellipticity, and microlocal phenomena in partial differential equations. 
\par Pseudo-differential operators were first explicitly defined by Kohn-Nirenberg \cite{Kohn} and H\"ormander \cite{Hormander} to connect singular integrals and differential operators. They can be seen as a generalization of the Fourier multiplier. For example, let us consider a multiplier operator $P$  acting on functions of the variable $x$ such that
\begin{eqnarray*}
P(e^{2\pi ix\cdot\xi}) = a(\xi)\, e^{2\pi ix\cdot\xi}, \qquad \xi \in \mathbb{R},
\end{eqnarray*}
where $a(\xi)$ known as a multiplier or symbol of $P$.  In general, the multiplier $a$ depends on both the position $x$ and frequency $\xi$.  That is,  
\begin{eqnarray}\label{C7-eq:1.1}
 P(e^{2\pi ix\cdot\xi}) = a(x,\xi)\, e^{2\pi ix\cdot\xi}.  
\end{eqnarray}
 If we replace the exponential function 
$e^{2\pi ix\cdot\xi}$ in \eqref{C7-eq:1.1}
 with a Schwartz function 
$f$ and then applying Fourier transform techniques, the equation \eqref{C7-eq:1.1} can be reformulated as \cite{Wong}:
\begin{equation} \label{C7-e1.1}
    (P_af)(x) = \int_{\mathbb{R}} a(x,\xi)\,\hat{f}(\xi)\,e^{ix
    \xi}\,d\xi,
\end{equation}
 where 
$a(x,\xi)$ lies in the symbol class $\mathcal{S}^m$ \cite{Wong}. The operator \eqref{C7-e1.1} is called the pseudo-differential operator. For $m\in \mathbb{R}$, and $\alpha$, $\beta \in \mathbb{N}$, the symbol class $\mathcal{S}^m$ consist of all $\mathcal{C}^\infty$ functions on $\mathbb{R}\times\mathbb{R}$, for which every pair of $(\alpha,\beta)$ such that there exist a positive constant $C_{\alpha,\beta}$ satisfying
\begin{equation*}
|\partial_x^\alpha\,\partial_\xi^\beta\,a(x,\xi)| \le C_{\alpha,\beta}\,(1+|\xi|)^{m-\beta}.
\end{equation*}
Using different types of symbol classes, the properties of pseudo-differential operators on the Schwartz space have been extensively studied by researchers such as H\"ormander, Taylor, Shubin, Zaidman, Wong, and Treves, etc \cite{Hormander, Wong, Zaidman, Cardona, Ruzhansky, Ruzhansky1}. H\"ormander proved the $L^2$ boundedness of the pseudo-differential operator for the symbol class $\mathcal{S}^m_{\rho,\delta}, \,\, 0\le \delta<\rho\le1$ \cite{HormanderL}.  Calder\'on and Vaillancourt discussed the $L^2$ boundedness of $P_a$, where $a\in \mathcal{S}^m_{\delta},\,\, 0\le \delta\le1$ \cite{Vaillancourt}. Later, Fefferman extended the boundedness result to $L^p$ spaces, for $1<p<\infty$ \cite{Fefferman}.

\par
Gradually,  the researchers have extended the theory in various directions. One significant direction is to replace the classical Fourier transform with alternative integral transforms \cite{Prasad1, Prasad}. In this context, Pathak developed a theory of pseudo-differential operators in the context of the Hankel transform and proved that these operators act as continuous linear mappings from the Zemanian space to itself \cite{Pandey}. In conjunction, Pathak and Upadhyay investigated the product and commutator of pseudo-differential operators and later established their $L^p$-boundedness \cite{ Upadhyay, Pathak}.
Further developments were made by Ben Salem and Dachraoui, who defined a pseudo-differential operator associated with the Jacobi differential operator \cite{Salem}. They have studied the boundedness of these operators for specific symbol classes and established an $L^1$-norm inequality. In 2000, Dachraoui discussed the pseudo-differential operator and the Sobolev space associated with the Dunkl transform, a generalization of the Hankel transform \cite{Dunkl1992}.  After decades,  Amri et al. studied the Calder\'on-Vaillancourt theorem and $L^p$ boundedness for the H\"ormander class associated with the same transform \cite{Abdelkefi}.
\par  Dunkl introduced a differential-difference operator, known as the Dunkl operator, by incorporating a reflection term into the classical differential operator \cite{Dunkl1992}. He considered an initial value problem associated with the Dunkl operator, which leads to a unique solution called the Dunkl kernel 
$E_k(x,y)$ \cite{soltani2004lp}. This kernel forms the base for defining an integral transform, known as the Dunkl transform, which has attracted growing interest due to its applications in various fields like signal analysis and image processing etc. It also retains fundamental properties of the classical Fourier transform, including Plancherel’s theorem, Parseval’s identity, the Riemann-Lebesgue lemma, and Hausdorff-Young's inequality \cite{soltani2004lp}.
\par In 2017, Ghazouani et al. introduced a linear canonical Dunkl transform (LCDT), by combining the Dunkl transform with a 
$2\times2$ matrix parameter \cite{ghazouani2017unified}. This generalization encompasses several classical transforms as special cases, such as the Hankel transform, the Dunkl transform, and the Fourier transform, among others \cite{Dunkl1992, Grafakos, Upadhyay}. In recent years, significant progress has been made in the study of the LCDT. Specifically, we have investigated various aspects of the theory, such as uncertainty principles \cite{USK1}, extremal functions and Calder\'on's reproducing formula \cite{USK3}, the real Paley-Wiener theorem \cite{USK2}, and localization operators \cite{USK}. Notably, the localization operator can be viewed as a pseudo-differential operator. Moreover, the pseudo-differential operator has wide applications in potential theory; in particular, Bessel and Riesz potentials are specific examples of pseudo-differential operators \cite{Grafakos}. The broad applicability of both the LCDT and pseudo-differential operators motivates developing a theory of pseudo-differential operators in the LCDT setting.
For $f\in \mathcal{S}(\mathbb{R})$, the pseudo-differential operator is defined by
\begin{equation*}
 P_{a,M}(f)(x) =  \int_{\mathbb{R}}  a(x,\lambda)\, \mathcal{D}_k^M(f)(\lambda)\, E_k^{M^{-1}}(x,\lambda)\,d\mu_k(\lambda),   
\end{equation*}
where $E_k^{M^{-1}}(x,\lambda)$ is an inverse kernel of the linear canonical Dunkl transform and $a(x,\lambda)$ lies in the symbol class $S_0^m$ (see Definition \ref{C7-D3.1}).
It serves as a generalization of \eqref{C7-e1.1}. The classical pseudo-differential operators are particularly well-suited for analyzing problems involving translation-invariant differential operators. However, in mathematical physics, representation theory, and harmonic analysis, many operators of interest are not translation invariant. Notable examples are the Dunkl operator and the linear canonical Dunkl operator, both of which do not exhibit translation invariance with respect to the classical translation operator.  This constitutes one of the key advantages of developing a pseudo-differential operator theory associated with the LCDT, in contrast to the classical theory.  Moreover, a pseudo-differential operator 
$P_{a,M}$ offers an appropriate analytical framework to study the partial differential operator associated with the linear canonical Dunkl operator.
\par The paper is organized as follows:
In Section \ref{C7-S2}, we introduce the basic notations and recall the definition of the linear canonical Dunkl transform. We also present some important properties of the LCDT and its associated differential operator. Section \ref{C7-S3} introduces the symbol classes and defines pseudo-differential operators in the framework of the linear canonical Dunkl transform. Several fundamental results are established, including the continuity of these operators on $\mathcal{S}(\mathbb{R})$ and an explicit representation of the LCDT acting on a pseudo-differential operator. In addition, we investigate the boundedness of these operators on the linear canonical Dunkl Sobolev spaces. In Subsection \ref{C7-SS3}, we derive an integral representation for the pseudo-differential operator and determine its adjoint. We then extend the definition to the space of tempered distributions and prove continuity from $\mathcal{S}'(\mathbb{R})$ into itself. 
Moreover, we discuss an alternative integral representation of the pseudo-differential operator, in which the kernel satisfies a pointwise bound on its partial derivatives. In addition, we prove the $L^2$ boundedness of the pseudo-differential operator for the symbol class $S^m_0$. Finally, Section \ref{C7-S4} discusses applications of the theory to non-homogeneous and nonlinear parabolic partial differential equations.

\section{Preliminaries} \label{C7-S2}
We first introduce some essential notational conventions to study the pseudo-differential operator in the framework of the linear canonical Dunkl transform.
\begin{itemize}
\item[-] $\mathcal{C}'(\mathbb{R})$ denotes the continuously differentiable functions on $\mathbb{R}$.\\
\item[-]$\mathcal{C}_c^\infty(\mathbb{R})$ is the collection of all infinitely differentiable functions having compact support.\\
\item[-]  $\mathcal{S}(\mathbb{R})$ is the collection of all smooth functions and all derivatives of the function decay at infinity faster than any polynomial growth.\\
\item[-]  $\mathcal{S}'(\mathbb{R})$ denotes the collection of tempered distributions on $\mathbb{R}$.\\

\item[-] Weighted function: For $k\ge -\frac{1}{2}$, the weight function $d\mu_k$ defined on $\mathbb{R}$ by 
 \begin{equation*}
     d\mu_k(x) = \frac{
 |x|^{2k+1}\, dx}{2^{k+1}\,\Gamma(k+1)}.
 \end{equation*} 
\item[-] Weighted $L^p$ space: For $1\le p <\infty$, we define the $L^p$ space associated with the weight function $d\mu_k$ as given by 
\begin{equation*}
 L_k^p(\mathbb{R}):= \left\{ f: \mathbb{R}\rightarrow \mathbb{C}\, \text{is a measurable function}: \int_{\mathbb{R}} |f(x)|^p\,d\mu_k(x) <\infty\right\}.
 \end{equation*}
\item[-] For the case of $p=\infty$, 
\begin{equation*}
\|f\|_{L^\infty_k(\mathbb{R})} = \mathop{\underset{x \in \mathbb{R}}{\text{ess. sup}}} |f(x)| < \infty.
\end{equation*}
\item [-]  We write $X \lesssim Y$ to indicate that $X \le C Y$, where $C$  is a positive constant that varies depending on the context of the inequality.
\end{itemize}
The Dunkl operator \cite{soltani2004lp}  is a generalization of the classical differential operator, playing a fundamental role in harmonic analysis associated with reflection groups and orthogonal polynomials. By adjoining a matrix parameter $M=(a,b;c,d)$ to the Dunkl operator $\Lambda_{k}$, we defined the linear canonical Dunkl operator $\Lambda_{k,M}$ in \cite{USK} by 
\begin{equation*}
    \Lambda_{k, M}(f)(x) = \Lambda_kf(x)-i\left(\frac{d}{b}\right) x\, f(x), \qquad f \in \mathcal{C}^1(\mathbb{R}),
\end{equation*}
where  $b\neq 0$, and  $|M|=1$.  This operator preserves the essential properties of the classical Dunkl operator. 
\begin{proposition} \label{C7-P:2.1}\cite{USK2}
\begin{enumerate}[$(i)$]
\item Let $f,g \in \mathcal{S(\mathbb{R})}$. Then we have 
\begin{equation*}
    \int_{\mathbb{R}} \Lambda_{k,M} f(x)\, \overline{ g(x)}\,  d\mu_k(x) = -\int_{\mathbb{R}} f(x)\,\overline{\Lambda_{k,M}
    g(x)}\,  d\mu_k(x).
\end{equation*}
\vspace{.2cm}
\item For $f\in \mathcal{ S(\mathbb{R})} $, we have 
\begin{equation}\label{C7-e:2.1}
\mathcal{D}_k^M \left( \Lambda_{k, M^{-1}}^nf\right)(\lambda) = \left( \frac{i\lambda}{b}\right)^n \, \mathcal{D}_k^M(f)(\lambda),
\end{equation}
for all $n \in \mathbb{N}^*$ where $\mathbb{N}^* = \mathbb{N}\cup \{0\}$.
\vspace{.2cm}
\item The space $\mathcal{S}(\mathbb{R})$ is invariant under the linear canonical Dunkl operator $\Lambda_{k,M}$.
\end{enumerate}
\end{proposition}
We formulated the initial value problem associated with $\Lambda_{k,M}$ as follows \cite{USK}: 
\begin{equation*}
 \left \{
 \begin{array}{ll}
 \Lambda_{k, M} f(x) =-i\left(\frac{\lambda}{b}\right) f(x),     & \mbox{}  \\
 f(0) = e^{\frac{i}{2} \frac{d}{b} \lambda^2}, &    \lambda \in \mathbb{R}.
 \end{array}
 \right.
\end{equation*}
It admits a unique analytic solution, denoted by $E_k^M(\lambda, x)$, which serves as the kernel for the linear canonical Dunkl transform. The explicit expression for $E_k^M(\lambda,x)$ is given by:
\begin{equation*}
E^M_{k}(\lambda,x) = e^{\frac{i}{2}(\frac{d}{b}\lambda^2+\frac{a}{b}x^2)}E_{k}(-i\lambda/b,x).
\end{equation*} 
The Dunkl kernel $E_k(i\lambda, x)$, as defined in \cite{soltani2004lp}, is given by:
 \begin{equation*}
E_k(i\lambda,x) = j_{k}(\lambda x)+\frac{i\lambda x}{2(k+1)}j_{k+1}(\lambda x).
\end{equation*}
The linear canonical Dunkl transform is defined for an integrable function
$f$ on $\mathbb{R}$ by incorporating $E^M_{k}(\lambda,x)$ \cite{ghazouani2017unified} as follows:
\begin{equation*}
\mathcal{D}_k^M(f)(\lambda) = \frac{1}{(ib)^{k+1}} \int_{\mathbb{R}}f(x)E^M_{k}(\lambda,x)\,d\mu_{k}(x),  \quad b \neq 0.
\end{equation*}
We focus on the case
$b\neq 0$, which yields a non-trivial transform. Conversely, when 
$b=0$, the transform reduces to a chirp multiplication followed by a dilation of 
$f$.
\begin{eqnarray*}
\mathcal{D}^M_{k}(f)(\lambda) = \frac{e^{i\frac{c}{2a}\lambda^2}}{|a|^{k +1}}f(\lambda/a).
\end{eqnarray*}
Throughout the paper, we will consider $b\neq 0$.
To better understand the operators $\mathcal{D}^M_{k}$ and $\Lambda_{k,M}$,  we outline their fundamental properties below \cite{ghazouani2017unified, USK}:
\begin{proposition}
\begin{enumerate} [$(i)$]
\item
\textbf{Plancherel's formula:} If $f \in  L^{1}_{k}(\mathbb{R})\cap L^{2}_{k}(\mathbb{R}) $, then  $\mathcal{D}^M_k(f) \in L^{2}_{k}(\mathbb{R})$ and  
 \begin{equation*}
 \|\mathcal{D}^M_k(f)\|_{ L^{2}_{k}(\mathbb{R})} = \| f \|_{ L^{2}_{k}(\mathbb{R})}.
 \end{equation*}
 \vspace{.2cm}
\item \textbf{Parsavel's formula:} For $f,g \in L_k^2(\mathbb{R})$, we have
\begin{equation*}
\int_{\mathbb{R}}f(x)\,\overline{g(x)}\,d\mu_k(x) = \int_{\mathbb{R}} \mathcal{D}_k^M(f)(\lambda)\,\overline{\mathcal{D}_k^M(g)(\lambda)}\,d\mu_k(\lambda).   
\end{equation*}
\vspace{.2cm}
\item \textbf{Inverse formula:}  For every $ f\in L^1_{k}(\mathbb{R})$ such that $\mathcal{D}^M_{k}f \in L^1_{k}(\mathbb{R})$, we have
\begin{equation*} 
\mathcal{D}^{M^{-1}}_{k} (\mathcal{D}^M_{k}f)
=f ,\,\, \text{a.e}.    
\end{equation*}
\vspace{.2cm}
 \item\textbf{Hausdorff-Young's inequality:} For $ 1 \le p \le 2$, the operator $\mathcal{D}_k^M$ is a bounded linear operator on $L_k^p(\mathbb{R})$, and it satisfies the following inequality:   \begin{equation*}
    \|\mathcal{D}_k^M(f)\|_{L_k^{q}(\mathbb{R})} \le \frac{1}{\left(|b|^{k+1}\right)^{1-\frac{2}{q}}}\, \|f\|_{L^p_k(\mathbb{R})},
\end{equation*}
where $q$ is the conjugate exponent of $p$.\\
\vspace{.2cm}
\item \textbf{Reimann-Lebesgue lemma:}
    For every $f\in L_k^1(\mathbb{R}),$ we have $\mathcal{D}_k^M(f)\in \mathcal{C}_0(\mathbb{R})$ and
    \begin{align*}
    \|\mathcal{D}_k^M(f)\|_{L_k^{\infty}(\mathbb{R})} \leq \frac{1}{|b|^{k+1}} \|f\|_{L_k^1(\mathbb{R})}.
    \end{align*}
\item The map $\mathcal{D}_k^M:\mathcal{S}(\mathbb{R}) \rightarrow \mathcal{S}(\mathbb{R})$ is a topological isomorphism.
\end{enumerate}
\end{proposition}
We shall recall the definition of the linear canonical Dunkl translation operator, which will be instrumental in establishing subsequent results.
\begin{definition} \cite{USK}
Let $f$ be a continuous function on $\mathbb{R}$. Then the translation operator is defined by
\begin{equation*}
 (\Tau^M_xf)(y)=  \int_{\mathbb{R}}e^{i\frac{a}{b}z^2}f(z)\,W^M_{k}(x,y,z)\,d\mu_k(z), 
\end{equation*}
where
\begin{equation*}
W^M_{k}(x,y,z) = e^{-\frac{i}{2}\frac{a}{b}(x^2+y^2+z^2)}\,W_{k}(x,y,z).    
\end{equation*}
\end{definition}
For each  $x,y \in \mathbb{R}$, we recall the inequality \cite{soltani2004lp}
\begin{equation} \label{C7-e2.2}
     \int_{\mathbb{R}}|W_k(x,y,z)|  d \mu_k(z) \le 4. 
     \end{equation}
One can easily check that \cite{soltani2004lp}
\begin{equation}
    W_k(x,y,z) =
\left \{
\begin{array}{lll}
 W_k(y,x,z), & \mbox{}\\
 W_k (-x,-y,-z),& \mbox{}\\
 W_{k}(x,-z,-y) . & \mbox{}
\end{array}  \right. \label{C7-eq: 2.2}
\end{equation} 
The linear canonical Dunkl transform acting on the translation operator is
given by
\begin{equation}\label{C7-e:2.2}
 D^M_{k}(\Tau^M_xf)(\lambda) =e^{-\frac{i}{2}\frac{a}{b}x^2} E_{k}(i\lambda/b,x)D^M_{k}(f)(\lambda).  \end{equation}
 The subsequent theorem examines the boundedness of the translation operators.
\begin{theorem}\label{C7-T :3.5} \cite{USK}
For $x\in \mathbb{R}$, the generalized translation operator $\Tau^M_x$ is a bounded linear operator on $L^p_{k}(\mathbb{R}), p \in [1,\infty]$ and satisfies
\begin{equation*}
 \|\Tau^M_xf\|_{L^{p}_{k} (\mathbb{R})} \le 4\,\|f\|_{L^{p}_{k} (\mathbb{R})}.
\end{equation*}
\end{theorem}
Next, we recall the convolution operator for the linear canonical Dunkl transform.\\
Let $f,g \in \mathcal{S}(\mathbb{R})$. Then the convolution operator is defined by 
\begin{equation*}
(f \underset{M}{\ast} g)(x) =  \frac{1}{(ib)^{k+1}}\int_{\mathbb{R}}e^{i \frac{a}{b}y^2}\,\Tau^M_xf(-y)\,  g(y)  \,d\mu_{k}(y).
\end{equation*} 
For all $f \in L^1_{k}(\mathbb{R})$  and $g\in L^2_{k}(\mathbb{R})$, we have the factorization identity 
\begin{equation} \label{Prop:2.2.4}
D^M_k(f \underset{M}{\ast} g)(\lambda) = e^{-\frac{i}{2}\frac{d}{b}\lambda^2} D^M_{k}(f)(\lambda) D^M_{k}(g)(\lambda).
\end{equation} 
\begin{theorem} 
If  $f\in L_k^p(\mathbb{R})$, $1\le p \le \infty$, and $g\in L_k^1(\mathbb{R})$, then $f \underset{M}{\ast} g$  is a bounded operator on $L_k^p(\mathbb{R})$ and we have
\begin{equation*} 
 \| f \underset{M}{\ast} g\|_{L_k^p(\mathbb{R})}  \le
\frac {4}{ |b|^{k+1}}\,\|f\|_{L^{p}_{k} (\mathbb{R})}\,\|g\|_{L^{1}_{k} (\mathbb{R})}.
\end{equation*}
\end{theorem}
Now, we introduce the Sobolev space associated with two parameters $s$ and $r$ in the setting of the linear canonical Dunkl transform.
\begin{definition}[Linear canonical Dunkl Sobolev spaces] \label{C7-D2.2}
Let $s \in \mathbb{R}$ and $1\le r <\infty$, we define the Sobolev space associated with the linear canonical Dunkl transform
$${\textbf{W}^{s,r}_{k,M}(\mathbb{R})} = \{ h \in \mathcal{S' (\mathbb{R}}) : \|(1+|\lambda|^2)^{\frac{s}{2}} \mathcal{D}^M_k(h)\|_{L_k^r(\mathbb{R})} < \infty\}. $$
For $r=\infty$, we have 
$$\textbf{W}^{s,\infty}_{k,M}(\mathbb{R})= \{ h \in \mathcal{S' (\mathbb{R}}) : \sup_{\lambda\in \mathbb{R}} \,(1+|\lambda|^2)^{\frac{s}{2}}\, |\mathcal{D}^M_k(h)(\lambda)| < \infty\}.$$
\end{definition}
\section{Pseudo-differential operator for LCDT} \label{C7-S3}
In this section, we introduce the pseudo-differential operator associated with the linear canonical Dunkl transform. We establish the preliminary results for the pseudo-differential operator, such as continuity on $\mathcal{S}(\mathbb{R})$, boundedness with respect to the linear canonical Dunkl Sobolev norm, and representation via an integral kernel.
\par Symbol classes play a crucial role in ensuring the rigorous definition. The researchers have introduced several types of symbol classes to address various analytical needs, among which the H\"ormander class is particularly prominent \cite{Hormander}. This class has been widely used to study the boundedness of the pseudo-differential operators on $L^p$
spaces, Sobolev spaces, and other functional spaces. In this work, we focus on the following symbol classes to define the pseudo-differential operators associated with LCDT:
\begin{definition} \label{C7-D:3.1}
 A map $ a: \mathbb{R}\times \mathbb{C}\rightarrow \mathbb{C}$ is said to belong to the symbol class $S^m$ if it satisfies the following conditions:
\begin{enumerate}[$(i)$]
    \item The function $a(x,\lambda)$ is infinitely differentiable with respect to both variables $x$ and $\lambda$.  
    \item [$(ii)$] For every $l,n,r \in \mathbb{N}$ and $m \in \mathbb{R}$, there exists a constant $C_{l,m,n,r}$ such that 
\begin{equation*}
    \left|(1+x^2)^r \,\partial^l_x\, \partial^n_\lambda\, a(x,\lambda) \right| \le C_{l,m,n,r} \, (1+|\lambda|^2)^{(m-n)/2}.
\end{equation*}
\end{enumerate}
\end{definition}
\begin{definition}\cite{Dachraoui} \label{C7-D3.1}
For $r=0$, the above definition reduces to the symbol class $S_0^m$.
\end{definition}

\begin{definition} \label{C7-D3.2}
Let  $a$ be a symbol in the class $S_0^m$. Then the pseudo differential operator $P_{a,M}$ on $ \mathcal{S}(\mathbb{R})$ is defined by
\begin{equation*}
P_{a,M}(f)(x) =  \int_{\mathbb{R}}  a(x,\lambda)\, \mathcal{D}_k^M(f)(\lambda)\, E_k^{M^{-1}}(x,\lambda)\,d\mu_k(\lambda), \quad \text{ for every} \quad f \in \mathcal{S}(\mathbb{R}).
\end{equation*}
\end{definition}
It is straightforward to verify that 
$P_{a,M}$ is well-defined. Analogous to the classical pseudo-differential operator, we prove that the pseudo-differential operator acts as a continuous linear mapping from the Schwartz space $\mathcal{S}(\mathbb{R})$ to itself. 
\begin{lemma}
Let $f\in\mathcal{S}(\mathbb{R})$ and $b\neq 0$. Define
\begin{eqnarray*}
    \tilde{f}(x) =  e^{\frac{i}{2}\frac{a}{b}x^2}f(x).
\end{eqnarray*}
Then the operator $P_{a,M}$ admits the representation
\begin{eqnarray*}
   P_{a,M}(f)(x) = K_{k,b}\, e^{-\frac{i}{2}\frac{a}{b}x^2}\, H_{a,M}(f)(x),
\end{eqnarray*}
where 
\begin{eqnarray*}
K_{k,b}&=& \frac{1}{(ib)^{k+1}\,b^{2k+2}},
\end{eqnarray*}
and
\begin{eqnarray*}
H_{a,M}(f)(x)&=& \int_{\mathbb{R}}  a(x,\lambda\,b)\, \mathcal{D}_k(\tilde{f})(\lambda)\, E_k(-ix,\lambda)\,d\mu_k(\lambda).  
\end{eqnarray*}
\end{lemma}
\begin{proof}
    Using the relation between the Dunkl transform and the linear canonical Dunkl transform, 
    \begin{eqnarray} \label{LCDT-remark}
 \mathcal{D}_k^M(f)(\lambda) 
 &=& \frac{e^{\frac{i}{2}\frac{d}{b}\lambda^2}}{(ib)^{k+1}}\, \mathcal{D}_k(\tilde{f})\left(\frac{\lambda}{b}\right),~b\neq 0,
\end{eqnarray}
we obtain
\begin{eqnarray*}
P_{a,M}(f)(x) &=& \frac{1}{(ib)^{k+1}} \int_{\mathbb{R}}  a(x,\lambda)\, e^{\frac{i}{2}\frac{d}{b}\lambda^2}\, \mathcal{D}_k(\tilde{f})\left(\frac{\lambda}{b}\right)\, E_k^{M^{-1}}(x,\lambda)\,d\mu_k(\lambda)
\end{eqnarray*}
Expanding the inverse LCDT kernel $E_k^{M^{-1}}$ and performing the change of variables $\lambda \mapsto b\lambda$, we arrive at
\begin{eqnarray*}
P_{a,M}(f)(x) &=& K_{k,b}\, e^{-\frac{i}{2}\frac{a}{b}x^2}\,\int_{\mathbb{R}}  a(x,b\,\lambda)\, \mathcal{D}_k(\tilde{f})(\lambda)\, E_k(-ix,\lambda)\,d\mu_k(\lambda),
\end{eqnarray*}
which yields the desired representation.
\end{proof}
\begin{corollary}\label{C7-T:3.3}
   Let $a$ be a symbol in the class $S_0^m$  and let $f\in \mathcal{S}(\mathbb{R})$. Then the operator $P_{a,M}$ defines a continuous linear mapping from $\mathcal{S}(\mathbb{R})$ into itself.
\end{corollary}
\begin{proof}
    By \cite{Dachraoui}, the operator $H_{a,M}$ is continuous on $\mathcal{S}(\mathbb{R})$. Moreover, multiplication by $e^{-\frac{i}{2}\frac{a}{b}x^2}$ defines a continuous automorphism on $\mathcal{S}(\mathbb{R})$. The result follows immediately.
\end{proof}

The following proposition establishes an estimate for the symbol and provides an explicit representation of the operator $P_{a,M}$ in terms of the linear canonical Dunkl transform.

\begin{proposition}\label{C7-p3.4}
 Let $a\in S^m$ and $f\in \mathcal{S}(\mathbb{R})$. Then
 \begin{itemize}
\item[$(i)$] there exist $C_{m,l}$ depends on $m$ and the non-negative integers $l$ such that
\begin{equation*}
\left|\mathcal{D}_k^M\left(e^{-\frac{i}{2}\frac{a}{b}(\cdot)^2}\,a(\cdot,y)\right)(\lambda) \right|\le \frac{C_{m,l}}{|b|^{2k+2}}\,(1+y^2)^{m/2}\,(1+(\lambda/b)^2)^{-l}.      
\end{equation*}
\item[$(ii)$] for all $\lambda \in \mathbb{R}$, we have
\begin{eqnarray*}
\mathcal{D}_k^M(P_{a,M}f)(\lambda) &=& \int_{\mathbb{R}}\int_{\mathbb{R}} \mathcal{D}_k^M\left(e^{-\frac{i}{2}\frac{a}{b}(\cdot)^2}\,a(\cdot,y)\right)(z)\,\mathcal{D}_k^M(f)(y)\\
 &&\times e^{-\frac{i}{2}\frac{d}{b}(z^2+\lambda^2-y^2)}\,W_k(\lambda,-y,z)\,d\mu_k(z)\,d\mu_k(y).
\end{eqnarray*}
\end{itemize}
\end{proposition}
\begin{proof}
$(i)$ Let $f\in\mathcal{S}(\mathbb{R})$. We will prove $(i)$ using the relation between the Dunkl transform and LCDT. That is
\begin{equation*}
\mathcal{D}_k^M(f)(\lambda) = \frac{ e^{\frac{i}{2}\frac{d}{b}\lambda^2}}{(ib)^{k+1}}\,\mathcal{D}_k(e^{\frac{i}{2}\frac{a}{b}(\cdot)^2}f)\left(\frac{\lambda}{b}\right).
\end{equation*}
We consider
\begin{eqnarray*}
\left|\mathcal{D}_k^M\left(e^{-\frac{i}{2}\frac{a}{b}(\cdot)^2}\,a(\cdot,y)\right)(\lambda) \right| = \frac{1}{|b|^{k+1}}\, \left|\mathcal{D}_k(a(\cdot,y))\left(\frac{\lambda}{b}\right)\right|.
\end{eqnarray*}
Using the result in  \cite[Proposition 5.1]{Dachraoui}, we will obtain the desired result.\\
$(ii)$ The result is established by employing the definitions of pseudo-differential operator \ref{C7-D3.2} and translation operator for the Dunkl transform \cite{USK}.
\begin{eqnarray*}
\mathcal{D}_k^M(P_{a,M}f)(\lambda) &=& \int_{\mathbb{R}}\int_{\mathbb{R}}a(x,y)\,\mathcal{D}_k^M(f)(y)\,E_k^{M^{-1}}(x,y)\,E_k^M(\lambda,x)\,d\mu_k(y)\,d\mu_k(x)\\
&=&\int_{\mathbb{R}}\int_{\mathbb{R}}a(x,y)\,\mathcal{D}_k^M(f)(y)\,e^{\frac{i}{2}\frac{d}{b}(y^2-\lambda^2)}\,E_k(ix/b,y)\\
&&\times E_k(-i\lambda/b,x)\,d\mu_k(y)\,d\mu_k(x)
\\
&=& \int_{\mathbb{R}}\int_{\mathbb{R}}a(x,y)\,\mathcal{D}_k^M(f)(y)\,e^{\frac{i}{2}\frac{d}{b}(y^2-\lambda^2)}\,\Tau_{\lambda}(E_k(-ix/b,\cdot))(-y)
\\
&&\times d\mu_k(y)\,d\mu_k(x)
\\
&=& \int_{\mathbb{R}}\int_{\mathbb{R}}\int_{\mathbb{R}}a(x,y)\,\mathcal{D}_k^M(f)(y)\,e^{\frac{i}{2}\frac{d}{b}(y^2-\lambda^2)}\,E_k(-ix/b,z)\\
&& \times W_k(\lambda,-y,z)\,d\mu_k(z)\,d\mu_k(y)\,d\mu_k(x).
\end{eqnarray*}
Using Fubini's theorem, we deduce that
\begin{eqnarray*}
\mathcal{D}_k^M(P_{a,M}f)(\lambda) &=&\int_{\mathbb{R}}\int_{\mathbb{R}} \left( \int_{\mathbb{R}}e^{-\frac{i}{2}\frac{a}{b}x^2}\,a(x,y)\,E_k^M(x,z)\,d\mu_k(x)\right)\\
&&\times e^{-\frac{i}{2}\frac{d}{b}(z^2+\lambda^2-y^2)}\,\mathcal{D}_k^M(f)(y)\,W_k(\lambda,-y,z)\,d\mu_k(z)\,d\mu_k(y).
\end{eqnarray*}
Hence, we attained the proof.
\end{proof}
The following results demonstrate the boundedness of pseudo-differential operators with respect to the linear canonical Dunkl Sobolev norm.
\begin{theorem}
Let $a \in S^m$ and $s\in \mathbb{R}$. Suppose there exists $C>0$ such that the following conditions hold:
\begin{enumerate}[$(i)$]
\item For every $f\in {\bf{W}}_{k,M}^{s,1}(\mathbb{R})$, we have
\begin{equation*}
\|P_{k,M}(f)\|_{{\bf{W}}_{k,M}^{0,\infty
}(\mathbb{R})}  \le C \,\|f\|_{{\bf{W}}_{k,M}^{s,1
}(\mathbb{R})}.
\end{equation*}
\item For every $f\in {\bf{W}}_{k,M}^{s,r}(\mathbb{R})$, $1\le r \le \infty$, we have
\begin{equation*}
\|P_{k,M}(f)\|_{{\bf{W}}_{k,M}^{0,r
}(\mathbb{R})}  \le C \,\|f\|_{{\bf{W}}_{k,M}^{s,r
}(\mathbb{R})}.  
\end{equation*}
\end{enumerate}
\end{theorem}
\begin{proof}
$(i)$ Given that $a\in S^m$ and $f\in  {\bf{W}}_{k,M}^{s,1}(\mathbb{R})$. From the Definition \ref{C7-D2.2}, \ref{C7-D3.1} and  Proposition \ref{C7-p3.4}, we have
\begin{eqnarray*}
    |\mathcal{D}_k^M(P_{a,M}f)(\lambda)| &\le& \frac{C_{m,0}}{|b|^{2k+2}}\,\int_{\mathbb{R}}(1+y^2)^{\frac{m}{2}}\, |\mathcal{D}_k^M(f)(y)|\\
    && \times \left(\int_{\mathbb{R}}|W_k(\lambda,-y,z)|\,d\mu_k(z) \right) d\mu_k(y).
\end{eqnarray*}
Using the bounds of second integration \eqref{C7-e2.2}, we obtain the inequality $(i)$.\\\\
$(ii)$ Let  $f\in {\bf{W}}_{k,M}^{s,r}(\mathbb{R})$, $1\le r \le \infty$. Applying Proposition \ref{C7-p3.4} and Minkowski's inequality, we obtain the following:
\begin{eqnarray*}
\int_{\mathbb{R}} |\mathcal{D}_k^M(P_{a,M}f)(\lambda)|^r\,d\mu_k(\lambda) &=& \int_{\mathbb{R}} \left|\int_{\mathbb{R}^2} \mathcal{D}_k^M\left(e^{-\frac{i}{2}\frac{a}{b}(\cdot)^2}\,a(\cdot,y)\right)(z)\,\mathcal{D}_k^M(f)(y) \right.\\\\
&&\times \left. e^{-\frac{i}{2}\frac{d}{b}(z^2+\lambda^2-y^2)}\,W_k(\lambda,-y,z)\,d\mu_k(z)\,d\mu_k(y)\right|^r\,d\mu_k(\lambda)\\\\
 &\le& \frac{C_{m,l}^r}{|b|^{r(2k+2)}} \int_{\mathbb{R}} \left( \int_{\mathbb{R}^2} (1+z^2)^{-l}\,(1+y^2)^{m/2}\,|\mathcal{D}_k^M(f)(y)| \right.\\\\
 &&\left. \times|W_k(\lambda,-y,z)|\,d\mu_k(z)\,d\mu_k(y) \right)^r d\mu_k(\lambda)\\\\
 &\lesssim&  \int_{\mathbb{R}^3}(1+z^2)^{-rl}\,(1+y^2)^{rm/2}\,|\mathcal{D}_k^M(f)(y)|^r\\\\
 && \times|W_k(\lambda,-y,z)|^rd\mu_k(\lambda)\,d\mu_k(z)\,d\mu_k(y).
 \end{eqnarray*}
 Using the properties of $W_k$ in \cite{soltani2004lp}, we calculate that
 \begin{eqnarray*}
\int_{\mathbb{R}} |\mathcal{D}_k^M(P_{a,M}f)(\lambda)|^r\,d\mu_k(\lambda)&\lesssim& \int_{\mathbb{R}^2}(1+z^2)^{-rl}\,(1+y^2)^{rm/2}\,|\mathcal{D}_k^M(f)(y)|^r\,d\mu_k(z)\,d\mu_k(y), 
 \end{eqnarray*}
 where $lr>2k+1$, which yields the required result.
\end{proof}
\subsection{Amplitude representation of the pseudo differential operator} \label{C7-SS3}

 We begin by expanding the definition of the pseudo-differential operator $P_{a,M}$ as
\begin{equation}\label{C7-eq:3.2}
P_{a,M}(f)(x) = \int_{\mathbb{R}}\int_{\mathbb{R}}a(x,\lambda)\, f(y)\,E_k^M(\lambda,y)\, E_k^{M^{-1}}(x,\lambda)\,d\mu_{k,b}(y)\,d\mu_k(\lambda),
\end{equation}
where $f\in \mathcal{S}(\mathbb{R})$ and $a(x,\lambda) \in \mathcal{S}^m$.
Although $f$ belongs to the Schwartz class, the absolute convergence of the above integral representation is not guaranteed a priori. To ensure that the operator $P_{a,M}$ is well defined, we approximate the symbol $a(x,\cdot)$ by smooth functions with compact support. More precisely, let $\gamma \in \mathcal{C}_c^\infty(\mathbb{R}\times\mathbb{R})$ such that $\gamma=1$ in a neighbourhood of the origin. For $0< \epsilon \le 1$, define  $a_\epsilon(x, \lambda) = a(x,\lambda)\, \gamma(\epsilon x,\epsilon\lambda)$. It is straightforward to verify that $a_\epsilon(x, \lambda)$ lies in $S^m$ for all $0< \epsilon \le 1$. By \cite[Proposition 2.1.10]{Ruzhansky}, we have that $a_\epsilon \to a$ pointwise as $\epsilon\to 0$, with uniform control for $0<\epsilon\le1$. Following the classical approach, we therefore define the pseudo-differential operator $P_{a,M}$ in the limiting sense: 
\begin{eqnarray}
 P_{a,M}(f)(x) = \lim_{\epsilon \to 0}\int_{\mathbb{R}}\int_{\mathbb{R}}a_\epsilon(x,\lambda)\, f(y)\,E_k^M(\lambda,y)\, E_k^{M^{-1}}(x,\lambda)\,d\mu_{k,b}(y)\,d\mu_k(\lambda).
\label{C7-eq3.1}
\end{eqnarray}
With this definition, the convergence of the integral in \eqref{C7-eq3.1} is well defined.
We now proceed to determine the adjoint $P^*_{a,M}$ of the pseudo-differential operator $P_{a,M}$.

 For $f,g \in \mathcal{S}(\mathbb{R})$, we define the adjoint operator as
\begin{equation*}
    \langle P_{a,M}(f),g  \rangle_{L_k^2(\mathbb{R})} = \langle f, P^*_{a,M}(g) \rangle_{L_k^2(\mathbb{R})}.
\end{equation*}
 From the representation \eqref{C7-eq3.1}, we infer the following expression for the adjoint operator $P^*_{a,M}$:
\begin{equation*}
P^*_{a,M}(g)(y)= \lim_{\epsilon \to 0}\int_{\mathbb{R}} \int_{\mathbb{R}} \overline{a_\epsilon(x,\lambda)} \,E_k^{M^{-1}}(\lambda,y)\,E_k^M(\lambda,x)\,g(x)\, d\mu_k(x)\,d\mu_{k,-b}(\lambda).
\end{equation*}
Under the same understanding of possibly non-convergent integrals as in \eqref{C7-eq:3.2}, the expression can be written as
\begin{equation*}
P^*_{a,M}(g)(y)= \int_{\mathbb{R}} \int_{\mathbb{R}} \overline{a(x,\lambda)} \,E_k^{M^{-1}}(\lambda,y)\,E_k^M(\lambda,x)\,g(x)\, d\mu_k(x)\,d\mu_{k,-b}(\lambda).
\end{equation*}

Using the same arguments as in Corollary \ref{C7-T:3.3}, we conclude that the adjoint operator $P^*_{a,M}: \mathcal{S}(\mathbb{R}) \to \mathcal{S}(\mathbb{R})$ is continuous.. 
\par We now extend the pseudo-differential operator $P_{a,M}$ from the Schwartz space $\mathcal{S}(\mathbb{R})$ to its dual space $\mathcal{S}'(\mathbb{R})$ in a natural way. First, we define the pseudo-differential operator in the distributional sense. Let 
$P_{a,M}:\mathcal{S}'(\mathbb{R}) \to \mathcal{S}'(\mathbb{R})$. 
For $u \in \mathcal{S}'(\mathbb{R})$ and $\phi \in \mathcal{S}(\mathbb{R})$, we set
\begin{eqnarray*}
(P_{a,M}u)(\phi)
&:=&
\int_{\mathbb{R}} P_{a,M}(u)(x)\,\phi(x)\,d\mu_k(x).
\end{eqnarray*}

If the distribution $P_{a,M}u$ is represented by a Schwartz function, then the above expression can be rewritten in terms of the $L_k^2$-inner product representation as
\begin{eqnarray*}
(P_{a,M}u)(\phi)
&=&
\langle P_{a,M}(u), \overline{\phi} \rangle_{L_k^2(\mathbb{R})} \\
&=&
\langle u, P_{a,M}^*(\overline{\phi}) \rangle \\
&=&
\int_{\mathbb{R}} u(x)\,
\overline{P_{a,M}^*(\overline{\phi})(x)}\,d\mu_k(x) \\
&=&
u\!\left(\overline{P_{a,M}^*(\overline{\phi})}\right).
\end{eqnarray*}

Hence, $(P_{a,M}u)(\phi)$ is well defined for every 
$u \in \mathcal{S}'(\mathbb{R})$ and $\phi \in \mathcal{S}(\mathbb{R})$. 
This leads to the following definition.

\begin{definition}\label{C7-D3.7}
Let $u \in \mathcal{S}'(\mathbb{R})$. The action of the pseudo-differential operator 
$P_{a,M}$ on tempered distributions is defined by
\begin{equation*}
(P_{a,M}u)(\phi)
:=
u\!\left(\overline{P_{a,M}^*(\overline{\phi})}\right),
\qquad 
\phi \in \mathcal{S}(\mathbb{R}).
\end{equation*}
This definition is well posed since $P_{a,M}^*$ acts continuously on 
$\mathcal{S}(\mathbb{R})$.
\end{definition}



\begin{proposition}
If $a \in S^m$ and $u\in \mathcal{S}'(\mathbb{R})$, then the operator $P_{a,M}:\mathcal{S}'(\mathbb{R}) \to \mathcal{S}'(\mathbb{R})$ is continuous.
\end{proposition}
\begin{proof}
Let us assume that $u_k \to u$ in $\mathcal{S}'(\mathbb{R})$. Then, we  consider $$ (P_{a,M}\,u_k)(\phi) = u_k\left(\overline{P^*_{a,M}(\overline{\phi})}\right),\quad \phi \in  \mathcal{S}(\mathbb{R}).$$
Since $u_k$ is a continuous linear functional on $\mathbb{R}$. This implies
$$ u_k\left(\overline{P^*_{a,M}(\overline{\phi})}\right) \to u\left(\overline{P^*_{a,M}(\overline{\phi})}\right).$$ It is clear from \ref{C7-D3.7} that $$(P_{a,M}\,u_k)(\phi)\to (P_{a,M}\,u)(\phi).$$ 
\end{proof}
\subsection{Kernel representation and $L^2$ boundedness of 
pseudo-differential operator}
In parallel with the classical theory of pseudo-differential operators, where the operator can be represented through an associated kernel, we now establish the corresponding kernel representation for the pseudo-differential operator in the framework of the linear canonical Dunkl transform. This formulation will play a crucial role in understanding the structural and mapping properties of the operator in the LCDT setting. 
\begin{definition}
Let $f\in \mathcal{S}(\mathbb{R})$. Then the kernel representation of the pseudo-differential operator is defined by    
\begin{eqnarray*}
 P_{a,M}(f)(x) =\int_{\mathbb{R}}f(y)\, \mathcal{K}^M(x,y)\,d\mu_k(y),    
\end{eqnarray*}
where
\begin{eqnarray*}
\mathcal{K}^M(x,y) &=&  \Tau_x^M(K(x,\cdot))(y), \quad\text{and}\\
K^M(x,z) &=& \mathcal{D}_k^M(e^{-\frac{i}{2}\frac{d}{b}(\cdot)^2}\,a(x,\cdot))(z).
\end{eqnarray*}
\end{definition}
Before establishing further results concerning this integral representation, we first verify that the kernel formulation of the pseudo-differential operator is consistent with Definition~\ref{C7-D3.2}.  Let us consider
\begin{eqnarray*}
\int_{\mathbb{R}}f(y)\, \mathcal{K}^M(x,y)\,d\mu_k(y) = \int_{\mathbb{R}}f(y)\,\Tau_x^M(K^M(x,\cdot))(y)\,d\mu_k(y)
\end{eqnarray*}
\begin{eqnarray*}
&=&\int_{\mathbb{R}^2} f(y)\,\mathcal{D}_k^M(e^{-\frac{i}{2}\frac{d}{b}(\cdot)^2}\,a(x,\cdot))(z)\,W_k^M(x,y,z)\,d\mu_k(z)\,d\mu_k(y)\\
&=& \int_{\mathbb{R}^3} f(y)\,e^{-\frac{i}{2}\frac{d}{b}\lambda^2}\,a(x,\lambda)\, E_k^M(\lambda,z)\,W_k^M(x,y,z)\,d\mu_{k,b}(\lambda)\,d\mu_k(z)\,d\mu_k(y).
\end{eqnarray*}
Using  Fubini's theorem, we deduce that
\begin{eqnarray*}
\int_{\mathbb{R}}f(y)\, \mathcal{K}^M(x,y)\,d\mu_k(y)&=& \int_{\mathbb{R}^2} e^{-\frac{i}{2}\frac{d}{b}\lambda^2}\,a(x,\lambda)\,E_k^M(\lambda,z)\\&& \times\left(\int_{\mathbb{R}}f(y)\, W_k^M(x,z,y)\,d\mu_k(y)\right)d\mu_{k,b}(\lambda)\,d\mu_k(z).
\end{eqnarray*}
Doing a change of variable yields that
\begin{eqnarray*}
\int_{\mathbb{R}}f(y)\, \mathcal{K}^M(x,y)\,d\mu_k(y) &=& \int_{\mathbb{R}^2} e^{-\frac{i}{2}\frac{d}{b}\lambda^2}\,a(x,\lambda)\,E_k^M(\lambda,z)\\&& \times\Tau_x^M(f)(z)\,d\mu_{k,b}(\lambda)\,d\mu_k(z).  
\end{eqnarray*}
Again, applying Fubini's theorem and \eqref{C7-e:2.2}, we arrive at
\begin{eqnarray*}
\int_{\mathbb{R}}f(y)\, \mathcal{K}^M(x,y)\,d\mu_k(y) &=&\int_{\mathbb{R}}a(x,\lambda)\,E^{M^{-1}}_{k}(x,\lambda)\,D^M_{k}(f)(\lambda)\,d\mu_k(\lambda) \\
&=& P_{k,M}(f)(x).
\end{eqnarray*}
 Next, we prove the decay property of $K^M(x,z)$.
\begin{proposition}
 Let $a(x,\cdot)$ be a symbol in $S^m$. Then there exist $r,n\in\mathbb{N}$ such that
 \begin{eqnarray*}
\left| \partial^l_x\,\partial^n_z\, K^M(x,z) \right| \le  C\,(1+x^2)^{-n}\,(1+z^2)^{-r} \quad \text{for}\quad l,r\in\mathbb{N}.
\end{eqnarray*}
\end{proposition}
\begin{proof}
Let $a(x,\cdot) \in \mathcal{C}_c^\infty(\mathbb{R})$ whose suppport contained in the set $I \subset \mathbb{R}$ and let us consider
 \begin{eqnarray*}
\partial^l_x\,\partial^n_z\,K^M(x,z)  &=& \int_{\mathbb{R}}  e^{-\frac{i}{2}\frac{d}{b}\lambda^2}\, \partial^l_x\,a(x,\lambda)\,(1+z^2)^{-r}\,(1+z^2)^{r} \partial^n_z\,E_k^M(\lambda,z)\,d\mu_k(\lambda) \\\\
\left|\partial^l_x\,\partial^n_z\,K^M(x,z) \right|&\le& (1+z^2)^{-r}\int_I \left| \partial^l_x\,a(x,\lambda)\right|\,(1+z^2)^{r} \left| \partial^n_z\,E_k^M(\lambda,z)\right|\,d\mu_k(\lambda).
\end{eqnarray*}
 The required result follows immediately from Definition \ref{C7-D:3.1} $(ii)$.
\end{proof}
 Now, we prove the $L^2$ boundedness of the pseudo-differential operator to the symbol class $S^m_0$. 
 \begin{theorem} \cite{Ruzhansky}
 If $T: \mathcal{S}'(\mathbb{R}) \to \mathcal{S}'(\mathbb{R})$ be a continuous linear operator such that $T(\mathcal{S}(\mathbb{R})) \subset L^2(\mathbb{R})$ and
 \begin{equation*}
     \|Tf\|_{L^2(\mathbb{R})} \le C\,\|f\|_{L^2(\mathbb{R})}, \quad \text{for some} \quad C>0,
 \end{equation*}
 for every $f\in \mathcal{S}(\mathbb{R})$. Then the bounded operator $T$ is extended on $L^2(\mathbb{R})$ to itself. 
 \end{theorem}
 Using the above theorem, we will prove the $L^2$ boundedness of $P_{a,M}$.
 \begin{theorem}  If $a\in S^m_0$, then the pseudo-differential operator 
 $P_{a,M}:L^2_k(\mathbb{R}) \to L^2_k(\mathbb{R})$ is a bounded operator.
\end{theorem}
\begin{proof}
We first prove this result for a symbol $a(x,\lambda)$, which is compactly supported with respect to the variable $x$. We can express $a(\cdot,\lambda)$ as
\begin{equation} \label{C7-e3.1}
 a(x,\lambda) = \int_{\mathbb{R}} \mathcal{D}_k^M(a(\cdot, \lambda))(\eta)\, E_k^{M^{-1}}(x,\eta)\, d\mu_{k,-b}(\eta).
\end{equation}
It is clear that  $a(\cdot, \lambda) \in \mathcal{S}(\mathbb{R})$. Therefore, using \eqref{C7-e:2.1} , we have 
\begin{equation*}
    \left(\frac{i\eta}{b}\right)^\alpha\, \mathcal{D}_k^M(a(\cdot, \lambda))(\eta) = \int_{\mathbb{R}} \Lambda^\alpha_{k,M^{-1}} a(x,\lambda)\,E_k^M(x,\eta)\,d\mu_{k,b}(x).
\end{equation*}
Also, we have 
\begin{eqnarray*}
 \mathop{\underset{\eta
 \in \mathbb{R}}{\text{sup}}}  \left|  \left(\frac{i\eta}{b}\right)^\alpha\, \mathcal{D}_k^M(a(\cdot, \lambda))(\eta) \right| \le C_\alpha\,\left(1+ \left|\frac{\eta}{b} \right|\right)^{-N}, \quad \forall \quad 
\alpha,  N \in \mathbb{N}.
\end{eqnarray*}
Using the Defintion \ref{C7-D3.2} and \eqref{C7-e3.1}, we have 
\begin{eqnarray*}
    P_{a,M}(f)(x) &=& \int_{\mathbb{R}} \int_{\mathbb{R}} E_k^{M^{-1}}(x,\lambda)\,E_k^{M^{-1}}(x,\eta)\,\mathcal{D}_k^M(a(\cdot,\lambda))(\eta)\,\mathcal{D}_k^M(f)(\lambda)\,\\ &&\times d\mu_{k,-b}(\eta)\, d\mu_k(\lambda)\\
    &=&\int_{\mathbb{R}}  (Sf)(\eta,x)\,d\mu_k(\eta),
\end{eqnarray*}
where
\begin{equation*}
(Sf)(\eta,x) = \int_{\mathbb{R}} E_k^{M^{-1}}(x,\lambda)\,E_k^{M^{-1}}(x,\eta)\,\mathcal{D}_k^M(a(\cdot,\lambda))(\eta)\,\mathcal{D}_k^M(f)(\lambda)\,d\mu_{k,-b}(\lambda).
\end{equation*}
Let us consider
\begin{eqnarray*}
    \int_{\mathbb{R}} |(Sf)(\eta,x)|^2\,d\mu_k(\eta) &\le&\int_{\mathbb{R}} |\mathcal{D}_k^M(a(\cdot,\lambda))(\eta)|^2\,|\mathcal{D}_k^M(f)(\lambda)|^2\,d\mu_{k,-b}(\lambda)\\
    &\le&  \mathop{\underset{\eta
 \in \mathbb{R}}{\text{sup}}}|\mathcal{D}_k^M(a(\cdot,\lambda))(\eta)|\,\|f\|_{L^2(\mathbb{R})}\\
 &\lesssim& \left(1+ \left|\frac{\eta}{b} \right|\right)^{-N}\,\|f\|_{L^2(\mathbb{R})}.
 \end{eqnarray*}
 Thus, we have
 \begin{eqnarray*}
 \|P_{a,M}f\|_{L_k^2(\mathbb{R})} &\le& \int_{\mathbb{R}}\|Sf(\eta,\cdot)\|_{L_k^2(\mathbb{R})}\,d\mu_k(\eta)\\
 &\lesssim& \int_{\mathbb{R}}\left(1+ \left|\frac{\eta}{b} \right|\right)^{-N}\,\|f\|_{L_k^2(\mathbb{R})}\,d\mu_k(\eta)\\
\|P_{a,M}f\|_{L_k^2(\mathbb{R})} &\lesssim&  \|f\|_{L_k^2(\mathbb{R})}.
 \end{eqnarray*}
 The result is proved for the compactly supported symbols. Next, we prove the result for the symbol $a$ which is not necessarily compactly supported.  For all $x_0 \in \mathbb{R}$ and $ N\ge0$,  there exist a constant $C_N(x_0
 )$ such that
 \begin{equation}\label{C7-e3.2}
\int_{|x-x_0| \le 1} |P_{a,M}(f)(x)|^2\,d\mu_k(x) \le  C_N(x_0
 )\, \int_{\mathbb{R}}\frac{|f(x)|^2}{(1+|x-x_0|^2)^N}\,d\mu_k(x).  
 \end{equation}
For the time being, we can assume that \eqref{C7-e3.2} is true and prove the boundedness of $P_{a,M}$.
\begin{eqnarray*}
    \int_{\mathbb{R}^2} \chi_{|x-x_0|\le 1}\, |P_{a,M}(f)(x)|^2\,d\mu_k(x)\,d\mu_k(x_0) &\lesssim& \int_{\mathbb{R}^2}  \frac{|f(x)|^2}{(1+|x-x_0|^2)^N}\,d\mu_k(x)\,d\mu_k(x_0), \end{eqnarray*}
    where $N >k+\frac{1}{2}$. Applying Fubini's theorem, we obtain that
\begin{eqnarray*}
l(\left[x-1,x+1\right])\,\int_{\mathbb{R}} |P_{a,M}(f)(x)|^2\,d\mu_k(x) &\lesssim& \int_{\mathbb{R}} |f(x)|^2\, d\mu_k(x)\\
&\lesssim& \int_{\mathbb{R}} |f(x)|^2\, d\mu_k(x).
\end{eqnarray*}
This completes the required proof. Now, let us 
  prove \eqref{C7-e3.2}. We start to prove this result for $x_0=0$ and then extend it to general $x_0\in \mathbb{R}$. The function $f=f_1+f_2$ can be written in terms of two functions $f_1$ and $f_2$ such that the support of $f_1$ contained in $\{x \in \mathbb{R}: |x|\le 3\}$ and support of $f_2$ contained in $\{x \in \mathbb{R}: |x|\ge 2\}$, respectively.  First, we prove the inequality  \eqref{C7-e3.2} for $f_1$.\\
Let $\gamma \in C_c^\infty(\mathbb{R})$ with $\gamma(x)=1$ on $|x|\le 1$. Then we define $\gamma(P_{a,M}) := (\gamma P)_{a,M}$ is a compactly supported pseudo-differential operator whose symbol is $\gamma(x)\,a(x,\lambda)$. We start with the estimation that
\begin{eqnarray}
\nonumber    \int_{-1}^{1} |P_{a,M}(f_1)(x)|^2\,d\mu_k(x) &=& \int_{\mathbb{R}} |
(\gamma P)_{a,M}(f_1)(x)|^2\, d\mu_k(x)\\
\nonumber  &\lesssim&    \int_{\mathbb{R}} |f_1(x)|^2\, d\mu_k(x) \\
  &\lesssim& \int_{-3}^3 |f(x)|^2\, d\mu_k(x).\label{C7-e:3.3}
\end{eqnarray}
Next, we shall prove the estimate for $f_2$. If $\{x\in \mathbb{R}:|x|\le 1\}$, then $x \notin \text{supp}\,f_2$. Therefore,  we have 
\begin{eqnarray*}
   | P_{a,M}(f_2)(x)| &\le& \int_{-2}^2 |f_2(y)|\, |\mathcal{K}^M(x,y)|\,d\mu_k(y)\\
    &\le&  \int_{\mathbb{R}}  |f(y)|\,|\Tau_x^M(K^M(x,\cdot))(y)|\,d\mu_k(y)
    \end{eqnarray*}
    Using Schwartz's inequality on the right-hand side of the integration, we get
    \begin{eqnarray*}
| P_{a,M}(f_2)(x)|    &\le& C\, \left(\int_{\mathbb{R}} \frac{|f(y)|^2}{(1+y^2)^N}\, d\mu_k(y)\right)^{\frac{1}{2}},
    \end{eqnarray*}
where 
\begin{equation*}
    C = \left(\int_{\mathbb{R}}(1+y^2)^N\,|\Tau_x^M(K^M(x,\cdot))(y)|^2\, d\mu_k(y)\right)^{\frac{1}{2}}.
\end{equation*}
Since $\Tau_x^M(K^M(x,\cdot))$ is a smooth function with suitable decay, therfore the above integral is converges.
Taking integration on both sides, we get
    \begin{eqnarray}\label{C7-e:3.4}
\int_{-1}^1  | P_{a,M}(f_2)(x)|^2\,d\mu_k(x)  &\lesssim&  \int_{\mathbb{R}} \frac{|f(y)|^2}{(1+|y|)^N}\, d\mu_k(y).
\end{eqnarray}
Combining \eqref{C7-e:3.3} and \eqref{C7-e:3.4} implies that \eqref{C7-e3.2}.
Now, we prove \eqref{C7-e3.2}   for an arbitrary $x_0\in \mathbb{R}$. We define $a_{x_0}(x,\lambda) = a(x-x_0, \lambda)$ and we see that the estimates of $P_{a,M}$ in $\{x\in \mathbb{R}:|x-x_0| \le 1\}$ is equivalent to  the same estimste of $P_{a_{x_0},M}$ in $\{x\in \mathbb{R}: |x|\le1\}$. This completes the proof of the theorem.
\end{proof} 

\section{Applications} \label{C7-S4}
In this section, we give some applications for the pseudo-differential operator associated with the linear canonical Dunkl transform. In general, pseudo-differential operators are generalization of linear partial differential operators. 
We aim to explore an application of $P_{a,M}$ formulated in terms of a partial differential operator involving the linear canonical Dunkl operator $\Lambda_{k,M}$. \\\\
$(i).$ We consider the following non-homogeneous partial differential-difference equation:
\begin{eqnarray*}
  (a_n\,\Lambda_{k,M^{-1}}^n+a_{n-1}\,\Lambda_{k,M^{-1}}^{n-1}+\cdots +a_1\,\Lambda_{k,M^{-1}}+a_0 )(\phi)(x) =f(x),
\end{eqnarray*}
where $\phi \in \mathcal{S}(\mathbb{R})$,  $f\in\mathcal{S}'(\mathbb{R})$ and $a_0,a_1, \cdots a_n$ are constants. Applying the linear canonical Dunkl transform on both sides and using \eqref{C7-e:2.1}, we have

\begin{eqnarray*}
\left[\ a_n\,\left(\frac{i\lambda}{b}\right)^n+a_{n-1}\,\left(\frac{i\lambda}{b} \right)^{n-1}+\cdots+a_0 \right]\  \mathcal{D}_k^M(\phi)(\lambda) = \mathcal{D}_k^M(f)(\lambda).
\end{eqnarray*}
Thus 
\begin{eqnarray*}
 \mathcal{D}_k^M(\phi)(\lambda) = \frac{\mathcal{D}_k^M(f)(\lambda)}{P\left(\frac{i\lambda}{b} \right)}, \qquad P\left(\frac{i\lambda}{b} \right)\neq 0, 
\end{eqnarray*}
where $P$ is a polynomial of degree $n$  with constant coefficients.
Applying the inverse linear canonical Dunkl transform to both sides yields the desired solution
\begin{eqnarray*}
\phi(x) = \mathcal{D}_k^{{M}^{-1}}\left[\frac{\mathcal{D}_k^M(f)(\lambda)}{P\left(\frac{i\lambda}{b} \right)}\right](x) .
\end{eqnarray*}
Let us consider the nonlinear parabolic partial differential equation given by
\begin{eqnarray*}
\frac{\partial}{\partial t
}\phi(x,t) = \Lambda_{k,M^{-1}}^2\phi(x,t)+F(\phi,x,t),
\end{eqnarray*}
where $\Lambda_{k,M^{-1}}$ is the Dunkl operator and $$F(\phi,x,t)=(\phi\underset{M}{\ast}  \phi)(x,t).$$
Applying LCDT on both sides with respect to $x$ and using \eqref{Prop:2.2.4}, we obtain 
\begin{equation}\label{C7-eq:4.1}
 \Phi_t(\lambda,t) = \left( \frac{i\lambda}{b}\right)^2\, \Phi(\lambda,t)+ e^{-\frac{i}{2}\frac{a}{b}\lambda^2} \Phi^2(x,t), 
\end{equation}
where $\Phi(x,t)$ is the LCDT of $\phi(x,t)$ with respect to $x$.  The equation \eqref{C7-eq:4.1} is a nonlinear Riccati equation \cite{Polyanin}, and its solution is obtained by substituting
\begin{equation}\label{C7-eq:4.2}
 \Phi(\lambda,t) = -\frac{e^{\frac{i}{2}\frac{a}{b}\lambda^2}\, w_t(\lambda,t)}{w(\lambda,t)}.   
\end{equation}
The above substitution leads to the following second-order linear differential equation
\begin{equation*}
    w_{tt}(\lambda,t)+\frac{\lambda^2}{b^2}\,w_t(x,t) =0.
\end{equation*}
By straightforward manipulation, we obtain
\begin{equation*}
    w(\lambda,t) = \beta_1(\lambda)\, e^{-\frac{\lambda^2}{b^2}t}+\beta_2(\lambda),
\end{equation*}
where $\beta_1(\lambda)$ and $\beta_2(\lambda)$ are constants depending on $\lambda$.
Substituting $w$ and $w_t$ in \eqref{C7-eq:4.2}, we deduce that
\begin{equation*}
    \Phi(\lambda,t) = \frac{\lambda^2\,e^{\frac{i}{2}\frac{a}{b}\lambda^2}}{b^2\left(1+\frac{\beta_2(\lambda)}{\beta_1(\lambda)}\,e^{\frac{\lambda^2}{b^2}t}\right)}, ~~\beta_1(\lambda)\neq 0.
\end{equation*}
By taking the inverse LCDT of
 $\Phi$, we obtain the required solution $\phi$.

\subsection*{Acknowledgements:} The second author acknowledges the funding received from Anusandhan National Research Foundation SURE (SUR/2022/005678).
\subsection*{Data availability:} No new data was collected or generated during this research.
\subsection*{Disclosure statement:} The authors report there are no competing interests to declare.  

\subsection*{Conflict of interest:} No potential conflict of interest was reported by the author.

\end{document}